\documentclass{amsart}

\usepackage{amssymb,amsmath}
\usepackage{mathrsfs}
\usepackage[neveradjust]{paralist}
\usepackage[all]{xypic}
\xyoption{dvips}
\usepackage{epic,eepic}
\usepackage{url}

\theoremstyle{definition}
\newtheorem{ntn}{Notation}[section]
\newtheorem{dfn}[ntn]{Definition}
\theoremstyle{plain}
\newtheorem{lem}[ntn]{Lemma}
\newtheorem{prp}[ntn]{Proposition}
\newtheorem{thm}[ntn]{Theorem}
\newtheorem{cor}[ntn]{Corollary}

\newtheorem{prb}[ntn]{Problem}
\theoremstyle{remark}

\newtheorem{rmk}[ntn]{Remark}
\newtheorem{exa}[ntn]{Example}

\newcommand{\boldzero}{{\mathbf{0}}}
\newcommand{\boldone}{{\mathbf{1}}}
\newcommand{\bolda}{{\mathbf{a}}}
\newcommand{\boldb}{{\mathbf{b}}}
\newcommand{\boldc}{{\mathbf{c}}}
\newcommand{\bsb}{{\boldsymbol b}}
\newcommand{\bolde}{{\mathbf{e}}}
\newcommand{\boldu}{{\mathbf{u}}}

\newcommand{\del}{\partial}
\newcommand{\eps}{\varepsilon}
\newcommand{\from}{\leftarrow}
\newcommand{\ideal}[1]{{\langle#1\rangle}}
\newcommand{\into}{\hookrightarrow}
\newcommand{\onto}{\twoheadrightarrow}
\newcommand{\xymat}{\SelectTips{cm}{}\xymatrix}
\newcommand{\sss}{s}
\newcommand{\calD}{\mathscr{D}}
\newcommand{\calE}{\mathscr{E}}
\newcommand{\calM}{\mathscr{M}}
\newcommand{\calO}{\mathscr{O}}
\newcommand{\frakm}{{\mathfrak{m}}}
\newcommand{\frakM}{{\mathfrak{M}}}
\newcommand{\frakS}{{\mathfrak{S}}}

\newcommand{\CC}{\mathbb{C}}
\newcommand{\NN}{\mathbb{N}}
\newcommand{\QQ}{\mathbb{Q}}
\newcommand{\ZZ}{\mathbb{Z}}

\DeclareMathOperator{\diag}{diag}
\DeclareMathOperator{\dlim}{\varinjlim}

\DeclareMathOperator{\gr}{gr}

\DeclareMathOperator{\id}{id}

\DeclareMathOperator{\qdeg}{qdeg}
\DeclareMathOperator{\Res}{Res}
\DeclareMathOperator{\SRes}{sRes}
\DeclareMathOperator{\Spec}{Spec}
\DeclareMathOperator{\Var}{Var}

\begin{document}

\title[Hypergeometric D-modules and twisted Gau\ss--Manin systems]
{Hypergeometric D-modules and\\twisted Gau\ss--Manin systems}

\author{Mathias Schulze}
\address{
M. Schulze\\
Dept.\ of Mathematics\\
Oklahoma State University\\
401 MSCS\\
Stillwater, OK 74078\\
USA}
\email{mschulze@math.okstate.edu}
\thanks{MS was supported by the College of Arts \& Sciences at Oklahoma State University through a FY08 Dean's Incentive Grant.}

\author{Uli Walther}
\address{
U. Walther\\
Purdue University\\
Dept.\ of Mathematics\\
150 N.\ University St.\\
West Lafayette, IN 47907\\
USA}
\email{walther@math.purdue.edu}
\thanks{UW was supported by the NSF under grant DMS~0555319, and by the NSA under grant H98230-06-1-0012.}

\begin{abstract}
The Euler--Koszul complex is the fundamental tool in the homological study of $A$-hypergeometric differential systems and functions. 
We compare Euler--Koszul homology with D-module direct images from the torus to the base space through orbits in the corresponding toric variety. 
Our approach generalizes a result by Gel'fand et al.\ \cite[Thm.~4.6]{GKZ90} and yields a simpler, more algebraic proof.

In the process we extend the Euler--Koszul functor to a category of infinite toric modules and describe multigraded localizations of Euler--Koszul homology.
\end{abstract}



\maketitle
\setcounter{tocdepth}{1}
\tableofcontents
\numberwithin{equation}{section}

\section{Introduction}

\subsection{Definition of GKZ-systems}

Let $\ZZ^d$ and $\ZZ^n$ denote the free $\ZZ$-modules with bases $\eps=\eps_1,\ldots,\eps_d$ and $e=e_1,\ldots,e_n$ respectively. 
Let $A=(a_{i,j})$ be an integer $d\times n$-matrix with columns $\bolda_1,\ldots,\bolda_n$. 
We consider $A$ both as a map $\ZZ^n\to\ZZ^d$ with respect to the bases above and as the finite subset $\{\bolda_1,\ldots,\bolda_n\}$ of $\ZZ^d$ consisting of the images of the $e_i$. 
We assume that $\ZZ A=\ZZ^d$ and that $\NN A$ is a positive semigroup which means that $0$ is the
only unit in $\NN A$.
To this type of data, Gel'fand, Graev, Kapranov and Zelevinski{\u\i} \cite{GGZ87,GKZ89}
associated in the 1980's a class of $D$-modules today called \emph{GKZ-} or \emph{$A$-hypergeometric systems} and defined as follows. 

Let $x_A=x_1,\dots,x_n$ be the coordinate system on $X:=\Spec(\CC[\NN^n])\cong \CC^n$
corresponding to $e$, and let $\del_A=\del_1,\dots,\del_n$ be the corresponding partial derivative
operators on the sheaf $\calO_X$ of regular functions on $X$ or its ring of global sections $\CC[x_A]$. 
Then the \emph{Weyl algebra}
\[
D_A=\CC\ideal{x_A,\del_A\mid[x_i,\del_j]=\delta_{i,j},\,[x_i,x_j]=0=[\del_i,\del_j]}
\]
is the ring of algebraic differential operators on $X=\CC^n$ and $D_A=D_X$ is
the ring of global sections of the sheaf $\calD_X$ of algebraic
differential operators on $X$.
With $\boldu_+=(\max(0,u_j))_j$ and $\boldu_-=\boldu_+-\boldu$, write
$\square_\boldu$ for $\del^{\boldu_+}-\del^{\boldu_-}$ where here and
elsewhere we freely use multi-index notation. The
\emph{toric relations of $A$} are then
\[
\square_A:=\{\square_\boldu \,\mid\, A\boldu=0\},
\]
while the \emph{Euler vector fields}
$E=E_1,\dots,E_d$ to $A$ are
\begin{equation}\label{42}
E_i:=\sum_{j=1}^na_{i,j}x_i\del_j.
\end{equation}
Finally, for $\beta\in\CC^d$, the \emph{$A$-hypergeometric system} is the $D_A$-module 
\[
M_A(\beta)=D_A/D_A\cdot\ideal{E-\beta,\square_A}.
\]

The structure of the solutions to the (always holonomic) modules
$M_A(\beta)$ is tightly interwoven with the combinatorics of the pair
$(A,\beta)\in(\ZZ^d)^n\times\CC^d$, and $A$-hypergeometric structures
are nearly ubiquitous. Indeed, research of the past two decades
revealed that toric residues, generating functions for intersection
numbers on moduli spaces, and special functions (Gau\ss, Bessel,
Airy, etc.) may all be viewed as solutions to GKZ-systems. In other
directions, varying Hodge structures on families of Calabi--Yau toric
hypersurfaces as well as the space of roots of univariate polynomials
with undetermined coefficients have $A$-hypergeometric structure.

\subsection{Torus action} 

Consider the algebraic $d$-torus $T=\Spec(\CC[\ZZ^d])$ with coordinate functions $t=t_1,\dots,t_d$ corresponding to $\eps=\eps_1,\ldots,\eps_d$.
One can view the columns $\bolda_1,\dots,\bolda_n$ of $A$, as
characters $\bolda_i(t)=t^{\bolda_i}$ on $T$, and the parameter vector
$\beta\in\CC^d$ as a character on its Lie algebra via $\beta(t_i\del_{t_i})=-\beta_i+1$.
A natural tool for investigating $M_A(\beta)$ is the torus action of $T$ on the cotangent space
$X^*=T^*_0X$ of $X$ at $0$ given by
\[
t\cdot\del_A=(t^{\bolda_1}\del_1,\dots,t^{\bolda_n}\del_n).
\]
The coordinate ring of $X^*$ is $R_A:=\CC[\del_A]$ which contains the \emph{toric ideal} $I_A$ generated by the toric relations $\square_A$.
For $\boldone_A=(1,\ldots,1)\in X$, $I_A$ is the ideal of the closure of the orbit $T\cdot\boldone_A$ of $\boldone_A$ whose coordinate ring is the \emph{toric ring} 
\[
S_A:=R_A/I_A\cong\CC[t^{\bolda_1},\dots,t^{\bolda_n}]\cong\CC[\NN A].
\]

The contragredient action of $T$ on $R_A$ given by
\[
(t\cdot P)(\del_A)=P(t^{-\bolda_1}\del_1,\dots,t^{-\bolda_n}\del_n),
\]
for $P\in R_A$, defines a $\ZZ^d$-grading on $R_A$ and on the coordinate ring $\CC[x_A,\del_A]$ of $T^*X$ by
\begin{equation}\label{38}
-\deg(\del_j)=\bolda_j=\deg(x_j).
\end{equation}
Note that for $\ZZ^d$-homogeneous $P\in R_A$ the commutator
$[E_i,P]$ equals $\deg_i(P)P$ where $\deg_i(-)$ is the $i$-th
component of $\deg(-)$. 
As  $\del_ix_i-x_i\del_i=1$, \eqref{38} also defines a $\ZZ^d$-grading on the sheaves of differential operators $\calD_X$ and $\calD_{X^*}$ under which $E$ and $\square_A$ are homogeneous.

Note that $A$ and $\beta$ naturally define an algebraic $\calD_T$-module
\begin{equation}\label{18}
\calM(\beta):=\calD_T/\calD_T\ideal{\del_{t_i}t_i+\beta_i\mid i=1,\dots,d},
\end{equation}
$\calO_T$-isomorphic to $\calO_T$ but equipped with a twisted
$\calD_T$-module structure expressed symbolically as
\[
\calM(\beta)=\calO_T\cdot t^{-\beta-1}
\]
on which $\calD_T$ acts via the product rule.

\subsection{Questions, results, techniques}

In our algebraic setting, the ring of global sections of $\calD_{X^*}$
is identified with $D_A$ via the Fourier transform. Under this
correspondence a natural question is the following:

\begin{prb}
Study the relationship between (the Fourier transform of)
$M_A(\beta)$ and the direct image $\phi_+\calM_\beta$ of
$\calM_\beta$ under the orbit map 
\[
\phi\colon T\to O\to X^*.
\]
\end{prb}

An important result in this direction was given in \cite[Thm.~4.6]{GKZ90}:
for non-resonant $\beta$ the two modules are
isomorphic. Here, a parameter is \emph{non-resonant} if it is not
contained in the locally finite subspace arrangement of resonant parameters
\[
\Res(A):=\bigcup_{\tau}\left(\ZZ^d+\CC\tau\right),
\]
the union being taken over all linear subspaces $\tau\subseteq \QQ^n$
that form a boundary component of the rational polyhedral cone $\QQ_+ A$.

\medskip

A powerful way of studying $M_A(\beta)$ is to consider it as a $0$-th
homology of a Koszul type complex $K_\bullet(S_A,\beta)$ of $E-\beta$
on $D_A/D_A\cdot\square_A\cong\CC[x_A]\otimes_\CC S_A$.  The idea
of such \emph{Euler--Koszul complex} is already visible in
\cite{GKZ89} and was significantly enhanced in \cite{MMW05}. Results
from \cite{MMW05} show that $K_\bullet(S_A,\beta)$ is a resolution of
$M_A(\beta)$ if and only if $\beta$ is not in the
\emph{$A$-exceptional locus} $\calE_A$, a well-understood (finite)
subspace arrangement of $\CC^n$.

In \cite[\S4]{Ado94} a variation of this complex can be found, whose is reminiscent of that of the Gau\ss--Manin system of the map
\[
(t_1,\dots,t_d)\mapsto(t^{\bolda_1},\dots,t^{\bolda_n})=\del_A
\]
by factorization through its graph, see for example \cite{Pha79}.
This suggests that $K_\bullet(S_A,\beta)$ might be suitable for representing the direct image $\phi_+(\calM_\beta)$, an observation (inspired by a talk by Adolphson) that became the catalyst for this
article.

The main results in this article, contained in Section~\ref{3}, make the relationship between $\phi_+(\calM_\beta)$, $M_A(\beta)$ and $K_\bullet(S_A,\beta)$ precise. 
We determine in Corollary~\ref{5} the exact set of parameters for which the first and last of these  agree. 
In view of \cite{MMW05}, this provides a considerable sharpening of \cite[Thm.~4.6]{GKZ90} expressing GKZ-systems in terms of twisted Gau\ss--Manin systems, stated in Corollary~\ref{5}.
In \cite[Thm.~2.11]{GKZ90}, this latter result is used in the homogeneous case to show that the generic monodromy representation on the (solution) space of $A$-hypergeometric functions is irreducible for non-resonant $\beta$.

The parameters identified in Corollary~\ref{5} are precisely those for which left-multiplication by $\del_i$ induces a quasi-isomorphism on $K_\bullet(S_A,\beta)$ for each $i=1,\dots,n$. 
We show in Corollary~\ref{36} that, given $\beta$, $K_\bullet(S_A,\beta+k\sum_{i=1}^n\bolda_i)$ has this property for $0\ll k\in\NN$. 
By Remark~\ref{24}, left-multiplication by $\del_i$ is a quasi-isomorphism on $K_\bullet(S_A,\beta)$ if and only if the \emph{contiguity operator} $\del_i\colon K_\bullet(S_A,\beta+k\bolda_i)\to K_\bullet(S_A,\beta+(k+1)\bolda_i)$ is a quasi-isomorphism for all $k\in\NN$ (which holds in general for $k\gg 0$). 
In this way, $K_\bullet(S_A[\del_A^{-1}],\beta)$ arises as the direct limit of the Euler--Koszul complexes $K_\bullet(S_A,\beta+k\sum_{i=1}^n\bolda_i)$ induced by the contiguity operators.

If $\beta$ is outside the set discussed in Corollary~\ref{5}, what is the ``difference'' between the Euler--Koszul complex $K_\bullet(S_A,\beta)$ and the direct image $\phi_+(\calM_\beta)$? 
In view of Corollary~\ref{36}, there is a natural (localization) map 
\begin{equation}\label{44}
K_\bullet(S_A,\beta)\to\phi_+(\calM_\beta)
\end{equation}
realized by a suitable (product of) contiguity operator(s). 
Example~\ref{23} suggests that there might be a filtration on the cone of \eqref{44} whose graded pieces consist of direct images of $\calM_\beta$ under $T$-orbit maps to border tori of $O$, that is, tori forming $\Bar O\smallsetminus O$.
In essence this would ask the following.

\begin{prb}\label{45}
Is there a relation between local cohomology of the Euler--Koszul complex $K_\bullet(S_A,\beta)$ with $T$-invariant support and direct images of $\calD_T$-modules $\calM_\beta$ under $T$-orbit maps to border tori of $O$?
\end{prb}

In the special case where $S_A$ is Cohen--Macaulay and hence $K_\bullet(S_A,\beta)$ is a resolution of $M_A(\beta)$, the local cohomology of hypergeometric systems was studied by Okuyama \cite{Oku06a}.
His main result \cite[Thm.~3.12]{Oku06a} shows indeed some similarity with Theorem~\ref{32} in Section~\ref{33} where we explicitly describe the direct images in Problem~\ref{45}.
Problem~\ref{45} also motivates to study general $\ZZ^d$-graded localizations of Euler--Koszul complexes.
This is the subject of Section~\ref{39}, where we generalize study the Euler--Koszul functor on direct limits of toric modules generalizing ideas from \cite{MMW05} and \cite{Oku06a}.

\section{Direct image via torus action}\label{17}

In this section we determine the direct image (complex) of $\calM_\beta$ under $\phi\colon T\to O\into X^*$. 
The orbit map $T\cong T\cdot\boldone_A=O$  identifies the coordinate ring $\CC[t^{\pm\bolda_1},\dots,t^{\pm\bolda_n}]\cong\CC[\ZZ^d]$ of $T$ with the coordinate ring $S_A[\del^{-1}_A]$ of $O$ where
$\del^{-1}_A:=\del_1^{-1},\dots,\del_n^{-1}$. 
The inclusion of the closure $\Bar O\subseteq X^*$ corresponds to the canonical projection $R_A\onto S_A$.

Put $E=E_A=E_1,\dots,E_d$ where $E_i$ is as in \eqref{42}. 
As $I_A$ is $\ZZ^d$-graded, each $E_i$, and hence $E_i-\beta_i$ as well, acts by right multiplication on representatives of classes in $D_A[\del_A^{-1}]/D_A[\del_A^{-1}]\cdot I_A=S_A[x_A,\del_A^{-1}]$. 
The total complex induced by these $d$ operators $E-\beta$ is the \emph{right Koszul complex}
induced by $E-\beta$. 

\begin{prp}\label{1}
The Fourier transform of the direct image $\phi_+\calM(\beta)$ is represented by the right Koszul complex of $E-\beta$ on $S_A[x_A,\del_A^{-1}]$ which is acyclic except in degree $0$.
\end{prp}

\begin{proof}
We use the abbreviation $\del_tt:=\del_{t_1}t_1,\dots,\del_{t_d}t_d$. 
Then the right Koszul complex of $-\del_tt-\beta$ on $\calD_T$ is a free $\calD_T$-resolution of $\calM(\beta)$.
Since $T$ is $\calD$-affine it suffices to check this on global sections.
But grading the global section complex with respect to the order filtration yields the Koszul complex of $-\del_tt$ on the Laurent polynomial ring in the variables $t^{\pm1}=t_1^{\pm 1},\dots,t_d^{\pm1}$ and $\del_t=\del_{t_1},\dots,\del_{t_d}$ which is clearly exact.

In order to compute the (Fourier transformed) direct image of this complex we factorize $\phi=\varpi\circ\iota$ into the closed embedding 
\begin{equation}\label{19}
\iota:T\into X^*\smallsetminus\Var(\del_1\cdots\del_n)=:Y^*
\end{equation}
and the open embedding 
\begin{equation}\label{20}
\varpi:Y^*\into X^*.
\end{equation}
The direct image $\iota_+\calM(\beta)$ is represented by the right Koszul complex of $-\del_tt-\beta$ on $\iota_+\calD_T=\iota_*\calD_{Y^*\gets T}$ where $\calD_{Y^*\gets T}$ is the transfer $(\iota^{-1}\calD_{Y^*},\calD_T)$-bimodule \cite[VI.5.1]{Bor87}.
Since $\iota$ is a closed embedding one can identify $\iota_+\calD_T$ and $\calD_{Y^*}/\calD_{Y^*}I_A$ as left $\calD_{Y^*}$-modules \cite[VI.7.3]{Bor87} and we have to verify that the right-action of $-\del_{t_i}t_i$ on this module translates into that of $E_i$ under this identification.
The transpose of $\calD_{Y^*\gets T}$ is the $(\calD_T,\iota^{-1}\calD_{Y^*})$-bimodule
\[
\calD_{T\to Y^*}=\iota^*\calD_{Y^*}=
\calO_T
\otimes_{\iota^{-1}\calO_{Y^*}}
\iota^{-1}\calD_{Y^*}
\]
whose left structure is given by the chain rule of differentiation \cite[VI.4.1]{Bor87}.
The transpose $t_i\del_{t_i}$ of $-\del_{t_i}t_i$ acts from the left on $\calD_{T\to Y^*}$ by
\[
t_i\del_{t_i}+\sum_{j=1}^nt_i\frac{\del t^{\bolda_j}}{\del t_i}\del_{\del_j}
=t_i\del_{t_i}-\sum_{j=1}^na_{i,j}t^{\bolda_j}x_j
\]
where $-x_j$ is considered as partial derivative $\del_{\del_j}$ with respect to $\del_j$ via the Fourier transform. 
Under the identification
\[
\iota_*\calD_{T\to Y^*}=\iota_*\calO_T\otimes_{\calO_{Y^*}}\calD_{Y^*}=\calD_{Y^*}/I_A\calD_{Y^*}
\]
this becomes $-\sum_{j=1}^na_{i,j}\del_jx_j$ whose transpose is $E_i$.
Therefore the right action of $\iota_+(-\del_{t_i}t_i)$ on $\iota_+\calD_T=\iota_*\calD_{Y^*\gets T}=\calD_{Y^*}/\calD_{Y^*}I_A$ coincides with that of $E_i$.
As $\iota$ is affine, the direct image functor $\iota_+$ is exact \cite[Prop.~VI.8.1]{Bor87}.
Thus the direct image $\iota_+\calM(\beta)$ is represented by the right Koszul complex of $E-\beta$ on $\calD_{Y^*}/\calD_{Y^*}I_A$ which is acyclic except in degree $0$.

The direct image functor $\varpi_+$ for the open embedding is the exact functor $\varpi_*$ \cite[VI.5.2]{Bor87}.
As $Y^*$ is affine and since $\Gamma(X^*,\varpi_*(\calD_{Y^*}/\calD_{Y^*}I_A))=S_A[x_A,\del_A^{-1}]$, $\phi_+\calM=\varpi_+\iota_+\calM(\beta)$ is represented by the acyclic right Koszul complex of $E-\beta$ on $S_A[x_A,\del^{-1}_A]$ as claimed.
\end{proof}

\begin{rmk}
We shall see an alternative proof for the acyclicity statement in Proposition \ref{1} in Remark~\ref{37}.\eqref{37a}.
\end{rmk}

\section{Euler--Koszul homology on localizations}\label{29}

In \cite{MMW05} a generalization of the right Koszul complex from Proposition~\ref{1} is developed as follows. 
Interpret right multiplication of $E_i$ on $S_A[x_A,\del_A^{-1}]$ as the effect of the left $D_A$-linear endomorphism $E_i$ that sends a $\ZZ^d$-homogeneous $y$ to
\begin{equation}\label{9}
E_i\circ y:=(E_i-\deg_i(y))y
\end{equation}
and extending $\CC$-linearly. 
The advantage of this point of view is that the definition extends verbatim to any left $D_A$-module $M$ with $\ZZ^d$-grading (\ref{38}). 

For a $\ZZ^d$-graded $R_A$-module $N$, the \emph{Euler--Koszul complex} $K_\bullet(N,\beta)=K^A_\bullet(N,\beta)$ of $N$ with parameter $\beta\in\CC^d$ is the Koszul complex of these (obviously commuting) endomorphisms $E-\beta$ on $D_A\otimes_{R_A}N$.  
The homology of this complex in the category of left $D_A$-modules $H_\bullet(N,\beta)=H^A_\bullet(N,\beta)$ is the \emph{Euler--Koszul homology}, a generalization of the $A$-hypergeometric system $M_A(\beta)=H^A_0(S_A,\beta)$. 
In this section we determine when $K_\bullet(S_A,\beta)$ is a representative for $\phi_+(\calM_\beta)$. By Proposition~\ref{1}, this amounts to describing when $K_\bullet(S_A,\beta)$ and $K_\bullet(S_A[\del_A^{-1}],\beta)$ are quasi-isomorphic.

\medskip

Throughout, we identify a submatrix $\tau$ of columns of $A$ with the corresponding set of column indices.
Denoting $\del^{k\tau}=\prod_{j\in\tau}\del_j^k$ and $\del_\tau^k=(\del_j^k)_{j\in \tau}$, $N[\del_\tau^{-1}]$ is $\ZZ^d$-graded if so is $N$. 
With this setup, the right action of $E_i$ on $S_A[x_A,\del_A^{-1}]=D_A\otimes_{R_A}S_A[\del_A^{-1}]$ in Proposition \ref{1} coincides with the left action (\ref{9}). 
For the present section, (\ref{9}) is more convenient.

By the left and right Ore property of $D_A$, the left and right localizations of $D_A$ coincide. 
In particular there is for any $R_A$-module $M$ an isomorphism
\begin{equation}\label{11}
R_A[\del_\tau^{-1}]\otimes_{R_A}(D_A\otimes_{R_A} M)\cong
D_A\otimes_{R_A}(M[\del_\tau^{-1}])
\end{equation}
of $D_A$-modules where $x_i$ acts on the left module by $x_i\cdot(P\otimes Q)=[x_i,P]\otimes Q+P\otimes(x_i\cdot Q)$.
This induces an isomorphism $R_A[\del_\tau^{-1}]\otimes_{R_A}K_\bullet(S_A,\beta)\cong K_\bullet(S_A[\del_\tau^{-1}],\beta)$ and, by exactness of localization, a corresponding isomorphism in homology.

\begin{cor}\label{2}
The complex $K_\bullet(S_A,\beta)$ represents $\phi_+\calM(\beta)$ if and only if left-multiplication by $\del_i$ is invertible on $H_\bullet(S_A,\beta)$ for $i=1,\dots,n$.\qed
\end{cor} 

Gel'fand et al.\ \cite[Thm.~4.6]{GKZ90} show that if $S_A$ is homogeneous and $\beta$ non-resonant for $A$ then the $A$-hypergeometric system $M_A(\beta)=H_0(S_A,\beta)$ represents $\phi_+\calM(\beta)$.
The non-resonance condition means that, for each proper face $F$ of the cone $\QQ_+A$, 
\[
\beta\not\in\CC F+\ZZ^d.
\]
By \cite[Prop.~5.3]{MMW05}, non-resonance implies that the Euler-Koszul complex is a resolution.
Homogeneity of $S_A$ is equivalent to $(1,\dots,1)$ being in the row span of $A$.

We shall describe the set of parameters $\beta$ for which
$K_\bullet(S_A,\beta)$ represents $\phi_+\calM(\beta)$ without any
homogeneity assumption.  At the same time we weaken the non-resonance
condition.  To do so we give a description of the parameter set in
$\CC^d$ for which left-multiplication by $\del_j$ is invertible on
$H_\bullet(S_A,\beta)$.

\begin{lem}\label{3}
Left-multiplication by $\del_j$ is injective on $D_A/D_AI_A$ for $j=1,\dots,n$.
\end{lem}

\begin{proof}
Consider the weight vector $L_A=(\boldone_A,\boldzero_A)$ where $\boldone_A=(1,\dots,1)\in\ZZ^n$ and $\boldzero_A=(0,\dots,0)\in\ZZ^n$.
Since $\del_j$ is $L_A$-homogeneous, it suffices to check the statement after grading with respect to
$L_A$. 
But $\gr^{L_A}(D_A/D_AI_A)=\CC[x_A]\otimes_\CC R_A/I_A$.
Thus, $\gr^{L_A}(D_A/D_AI_A)$ is a domain and multiplication by $\del_j$ injective.
\end{proof}

In \cite[Def.~5.2]{MMW05}, the following notion was introduced.

\begin{dfn}\label{41}
For a finitely generated $\ZZ^d$-graded $R_A$-module $M$, the set of \emph{quasi-degrees} $\qdeg(M)$ is the Zariski closure of the set $\deg(M)$ of all $\alpha\in\ZZ^d$ for which $M_\alpha\ne0$.
\end{dfn}

We can now formulate an important class of parameters.

\begin{dfn}
Let
\[
\SRes_j(A):=\{\beta\mid
-\beta\in(\NN+1)\bolda_j+\qdeg(S_A/\ideal{\del_j})\};
\]
we let $\SRes(A):=\bigcap_{j=1}^n\SRes_j(A)$ be the \emph{strongly resonant parameters of $A$}.
\end{dfn}
\begin{thm}\label{6}
For $j=1,\dots,n$ the following conditions are equivalent:
\begin{enumerate}
\item\label{6a} $\beta\notin\SRes_j(A)$
\item\label{6b} Left-multiplication by $\del_j$ is a quasi-isomorphism
on $K_\bullet(S_A,\beta)$.
\end{enumerate}
\end{thm}

\begin{proof}
Without loss of generality we may assume $j=n$. 
By nature of the Euler--Koszul complex and Lemma~\ref{3}, left-multiplication by $\del_n$ defines a chain map
\begin{equation}\label{27}
\del_n\cdot\colon K_\bullet(S_A,\beta)\to K_\bullet(S_A,\beta)
\end{equation}
We prove that its cokernel $C'_\bullet$, and thus by Lemma~\ref{3} its cone, is exact precisely if $\beta$ is not strongly resonant for $A$.

Let $A'$ be obtained from $A$ by deleting the last column $\bolda_n$, set $E'=E_{A'}$ and note that $S_A/\ideal{\del_n}$ is a toric $R_{A'}$-module.
Each $C'_k$ is a direct sum of copies of the $D_{A'}$-module 
\[
D_A/(D_AI_A+\del_nD_A)\cong
D_{A'}\otimes_{R_{A'}}S_A/\ideal{\del_n}\otimes_\CC\CC[x_n]. 
\]
Note that the element $1\otimes(1+\ideal{\del_n})\otimes x_n^k$ lies in degree $k\bolda_n$ as $E$ is $\ZZ^d$-homogeneous of degree zero. 
In order to compute the induced differential in $C'_\bullet$, pick $\ZZ^d$-homogeneous $P'\in D_{A'}$ and $Q'\in S_A/\ideal{\del_n}$. 
Then for $k\in\NN$ one computes
\begin{align*}
E_i\circ(P'\otimes Q'\otimes x_n^k)
&=(E'_i+a_{i,n}x_n\del_n-\deg_i(P'\otimes Q'\otimes x_n^k))\cdot(P'\otimes Q'\otimes x_n^k)\\
&=(E'_i-\deg_i(P'\otimes Q')-(k+1)a_{i,n})\cdot(P'\otimes Q'\otimes x_n^k)\\
&=(E'_i-(k+1)a_{i,n})\circ(P'\otimes Q'\otimes x_n^k)
\end{align*}
It follows that $C'_\bullet$ decomposes as the sum of $D_{A'}$-complexes
\[
C'_\bullet=\bigoplus_{k\ge0}K^{A'}_\bullet(S_A/\ideal{\del_n},\beta+(k+1)\bolda_n)x_n^k.
\]
By \cite[Prop.~5.3]{MMW05}, exactness of the $k$-th summand is
equivalent to $-\beta\notin(k+1)\bolda_n+\qdeg(S_A/\ideal{\del_n})$
and the equivalence follows.
\end{proof}

\begin{rmk}
\label{24}
Consider the cokernel $C_\bullet$ of the (injective) chain map
\begin{equation}\label{25}
\cdot\del_n\colon K_\bullet(S_A,\beta)\to K_\bullet(S_A,\beta+\bolda_n)
\end{equation}
induced by $\del_n$ acting by right-multiplication on $S_A$.  
The modules of $C_\bullet$, considered as 
$D_{A'}$-modules, have a direct sum decomposition equal
to those of the complex $C'_\bullet$ that appears in the proof of
Theorem \ref{6} as cokernel of left-multiplication by $\del_n$ on
$K_\bullet(S_A,\beta)$.  However, the differentials in $C_\bullet$ and
$C'_\bullet$ are not the same: the differential in $C_\bullet$ is
$D_A$-linear, while that of $C'_\bullet$ is only $D_{A'}$-linear. 
It follows from \cite{MMW05} that \eqref{25} is a quasi-isomorphism if
and only if $-\beta\notin\bolda_n+\qdeg(S_A/\ideal{\del_n})$.  
\end{rmk}

The following corollary is the promised sharpening of \cite[Thm.~4.6]{GKZ90}; it determines when the hypergeometric module $M_A(\beta)$ is isomorphic to $\phi_+\calM(\beta)$.
\begin{cor}\label{5}
The following are equivalent:
\begin{enumerate}
\item\label{5a} $\beta\notin\SRes(A)$;
\item\label{5b} $K_\bullet(S_A,\beta)$ represents $\phi_+\calM(\beta)$;
\item\label{5c} $M_A(\beta)$ is naturally isomorphic to $\phi_+\calM(\beta)$. 
\end{enumerate}
\end{cor}

\begin{proof}
By Theorem~\ref{6}, \eqref{5a} is equivalent to left-multiplication of any $\del_j$ on $H_\bullet(S_A,\beta)$ being an isomorphism. 
By Corollary~\ref{2}, this is equivalent to 
\[
H_\bullet(S_A,\beta)\cong H_\bullet(S_A[\del_A^{-1}],\beta)\cong
 H_\bullet(\phi_+\calM(\beta))
\]
and hence to \eqref{5b}. 
It remains to show the equivalence with \eqref{5c}.

If \eqref{5b} holds then left-multiplication by $\del_A$ is a quasi-isomorphism on $K_\bullet(S_A,\beta)$ by Corollary~\ref{2}.
Since higher Euler--Koszul homology is $\del_A$-torsion it must vanish in this case. 
It follows that when \eqref{5b} holds then $M_A(\beta)$ is quasi-isomorphic to $K_\bullet(S_A,\beta)$ and hence to $\phi_+\calM(\beta)$. 
Thus, \eqref{5b} implies \eqref{5c}.

Conversely, if \eqref{5c} holds, then $\phi_+\calM(\beta)$ is a resolution of $M_A(\beta)$. 
Since left-multiplication by $\del_A$ is always a quasi-isomorphism on $\phi_+\calM(\beta)$ by Proposition~\ref{1}, $\del_A$ is invertible on $M_A(\beta)$. 
For the cokernel $C'_\bullet$ of \eqref{27} we have then $H_0(C'_\bullet)=0$. 
Since $C'_\bullet$ decomposes into a sum of Euler--Koszul complexes over $D_{A'}$, vanishing of $H_0(C'_\bullet)=0$ is equivalent to vanishing of $H_\bullet(C'_\bullet)=0$ by \cite[Prop.~5.3]{MMW05}. 
It follows that \eqref{27} is a quasi-isomorphism and by Theorem~\ref{6} we conclude
that \eqref{5a} holds.
\end{proof}

It is natural to ask whether there are any parameters that satisfy the hypothesis of Corollary~\ref{5}. 
To answer this question, we denote 
\[
\eps_\tau:=-\deg(\del^\tau)=\sum_{j\in \tau}\bolda_j
\]
for any $\tau\subseteq A$.
Note that multiplication by the invertible function $t_i$ on $T$ defines an isomorphism of $\calD_T$-modules $\calM(\beta)\to\calM(\beta+\bolde_j)$ shifting the degree by $\bolda_j$.

\begin{cor}\label{8}
For fixed $\beta\in\CC^d$ and $k\gg0$, 
\[
\phi_+\calM(\beta)\simeq\phi_+\calM(\beta+k\eps_A)\simeq K_\bullet(S_A,\beta+k\eps_A)\simeq M_A(\beta+k\eps_A).
\]
\end{cor}

\begin{proof}
Obviously, elements $\beta\in\QQ_+A\subseteq\QQ^d$ satisfy the condition of Corollary~\ref{5} and, for any $\beta\in\QQ^d$, we have $\beta+k\eps_A\in\QQ_+A$ for $k\gg0$.
\end{proof}

For $k\gg 0$ the contiguity operator $\cdot \del_j$ induces a
quasi-isomorphism between the complexes $K_\bullet(S_A,\beta+k\bolda_j)$ and
$K_\bullet(S_A,\beta+(k+1)k\bolda_j)$ by
Remark~\ref{24}. Thus, 
\begin{eqnarray*}
\dlim K_\bullet(S_A,\beta+k\bolda_j)&\simeq&
K_\bullet\left(\dlim(S_A\stackrel{\del_j}{\to}S_A\stackrel{\del_j}{\to}S_A\cdots),\beta\right)\\
&\simeq&K_\bullet(S_A[\del_j^{-1}],\beta)
\end{eqnarray*}
for $k\gg 0$. 
Hence, for $k\gg 0$, 
$\phi_+\calM(\beta)\simeq K_\bullet(S_A,\beta+k\eps_A)\simeq
K_\bullet(S_A[\del_A^{-1}],\beta)\simeq R_A[\del_A^{-1}]\otimes_{R_A}
K_\bullet(S_A,\beta)$. 

\begin{cor}\label{36}
For fixed $\beta\in\CC^d$ and $k\gg0$, 
\[\pushQED{\qed}
R_A[\del_A^{-1}]\otimes_{R_A}H_\bullet(S_A,\beta)=H_\bullet(S_A,\beta+k\eps_A)(k\eps_A).\qedhere
\]
\end{cor}

\section{Direct images through border tori}\label{33}

For any $A\in\ZZ^{d\times n}$ and $\beta\in\CC^d$ there is a natural localization map
\[
K_\bullet(S_A,\beta)\to
K_\bullet(S_A[\del_A^{-1}],\beta)\simeq\phi_+(\calM_\beta),
\]
which for most $\beta$ is an isomorphism according to Corollary~\ref{5}. 
One wonders what the cone of this map is when it is not an isomorphism. 
Its Fourier transform must be supported in $\Bar O\smallsetminus O$ since both complexes in question
are supported in $\Bar O$ and agree on $O$.

\begin{exa}\label{23}
Consider the case $A=\begin{pmatrix}1&1&0\\0&1&1\end{pmatrix}$ with $\beta=(1,-1)$. 
According to our results above, $\beta$ is in the set of strongly resonant parameters sketched below and $\phi_+(\calM(\beta))\simeq M_A(\gamma)$ for $\gamma\in\beta+(\NN+1)\bolda_2$. 
A calculation with {\sl Macaulay2} \cite{M2} shows 
that the cone $C'_\bullet$ 
over the localization map $M_A(\beta)\to M_A(\beta+\bolda_2)$ has homology $H^0(C'_\bullet)\cong H^1(C'_\bullet)\cong D_A/\ideal{x_1\del_1-2,\del_2,\del_3}$. 

\begin{figure}[ht]
\begin{center}
\caption{$A$ and $\SRes(A)$ (dotted) in Example~\ref{23}}\label{26}
\setlength{\unitlength}{0.5mm}
\begin{picture}(80,80)(-40,-40)
\put(-40,0){\vector(1,0){80}}
\put(0,-40){\vector(0,1){80}}
\put(10,0){\circle*{2}}
\put(0,10){\circle*{2}}
\put(10,10){\circle*{2}}
\put(10,-10){\circle{2}\makebox(0,0)[tl]{$\beta$}}
\put(20,0){\circle{2}\makebox(0,0)[bl]{$\gamma$}}
\multiput(-10,0)(-10,0){3}{\dottedline{2}(0,-35)(0,35)}
\multiput(0,-10)(0,-10){3}{\dottedline{2}(-35,0)(35,0)}
\end{picture}
\end{center}
\end{figure}
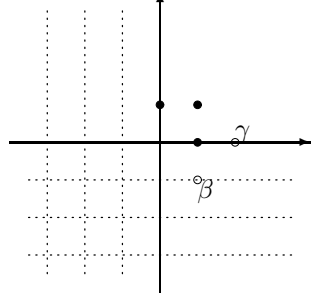

Consider the projection $\Spec(\CC[\ZZ^2])\to\Spec(\CC[\ZZ\bolda_1])$; corresponding direct images are computed by the left Koszul complex induced by $t_2\del_{t_2}$. 
Combine this projection with the embedding $\Spec(\CC[\ZZ\bolda_1])\to \CC^3$ sending $t_1$ to $(t_1,0,0)$ induced by the face $\QQ_+\bolda_1$ of $\QQ_+A$. 
The direct image of $\calM(\beta)$ under the composition has the same cohomology as $C'_\bullet$. Note that $\ZZ\bolda_1$ also ``causes'' $\beta$ to be in $\SRes(A)$.
\end{exa}

Example~\ref{23} suggests the investigation of direct images of $\calM(\beta)$ under orbit maps $T\to O$
factoring through tori in $\Bar O\smallsetminus O$, as well as partial $\ZZ^d$-graded localizations of Euler--Koszul homology of $S_A$. 
In the final two sections we follow this line of thought: we determine the structure of these direct
images in Theorem~\ref{32}, and we give a result that mirrors Corollary~\ref{36} in Corollary~\ref{12}.

\medskip

By abuse of language, we call a submatrix $F$ of columns of $A$ a \emph{face} of $A$ if $\QQ_+F$ is a face of the cone $\QQ_+A$.
The toric variety $V_A=\Spec(S_A)=\Var(I_A)\subseteq X^*=X^*_A$ is the closure of the orbit $O=O_A$ through $\boldone=\boldone_A$. 
The complement $\Var(I_A)\smallsetminus O_A$ is a union of other orbits, 
\[
\Var(I_A)=\bigsqcup_F O_F,
\]
the union being taken over the faces $F$ of $A$.
Here, $O_F$ is the orbit of $\boldone_F\in X^*_A$ where $(\boldone_F)_i=1$ for $i\in F$ and $(\boldone_F)_i=0$ for $i\in\Bar F:=A\smallsetminus F$. 
As is well-known, the closure of $O_F$ corresponds to the $\ZZ^d$-graded prime $I_A^F=R_A\ideal{I_F,\del_{\Bar F}}$ of $R_A$. 
Note that $\del_j\in I^F_A$ for no $j\in\Bar F$.

Denote by $\phi_F\colon T_A\to X^*_A$ the $T_A$-equivariant map induced by sending $1\in T_A$ to $\boldone_F\in T^*_A$. 
Denote $X_F=\CC^{|F|}\subseteq X_F$, $X^*_F=T^*_0X_F$ and $Y^*_F=X^*_F\smallsetminus
\Var(\prod_F\del_j)$. 
Then $\phi_F$ has a natural factorization as follows. 
\begin{equation}\label{21}
\xymat{\phi_F\colon
T_A\ar@{->>}[r]^-{\pi_F}&
T'_F\ar@{->>}[r]^-{\gamma_F}&
T_F\ar@{^(->}[r]^-{\iota_F}&
Y^*_F\ar@{^(->}[r]^-{\varpi_F}
&X^*_F\ar@{^(->}[r]&X^*_A}.
\end{equation}
Indeed, let
\[
Q'_F:=\QQ F\cap \ZZ^d\supseteq \ZZ F
\]
be the saturation of $\ZZ F$ in the ambient lattice. 
Abstractly, $\ZZ F$ and $Q'_F$ are isomorphic, both being free Abelian groups of rank $\dim F$.
Concretely, the inclusion $Q'_F\supseteq\ZZ F$ is of finite index and so there is a basis for $Q'_F$ together with a diagonal matrix $K=\diag(k_1,\ldots,k_{\dim F})$ such that, in this basis, $\ZZ F$ is generated by the image of $K$.

Then $\phi_F$ is the map of semigroup ring spectra associated to the composition of semigroup morphisms
\[
\xymat{
\ZZ A&\ar@{_(->}[l]
Q'_F&\ar@{_(->}[l]
\ZZ F&\ar@{->>}[l]
\ZZ^{|F|}&\ar@{_(->}[l]
\NN^{|F|}&\ar@{->>}[l]
\NN^{|A|}.}
\]
In particular, in (\ref{21}),
\begin{itemize}
\item the rightmost map is induced by $\del_j\mapsto 0$ for $j\in\Bar F$;
\item $\varpi_F$ and $\iota_F$ are defined like their pendants $\varpi=\varpi_A$ and $\iota=\iota_A$, by decorating each letter in (\ref{20}) and in (\ref{19}) with a subscript ``$F$'';
\item $\gamma_F$ is the finite covering map induced by $K$;
\item $\pi_F$ is the natural projection of tori.
\end{itemize}

Generalizing Section \ref{17} we shall determine now the direct image $(\phi_F)_+\calM(\beta)$ of $\calM(\beta)$.
As in the special case $F=A$, this direct image agrees essentially with (multiple copies of) a localized Euler--Koszul complex.
In the computation of $(\phi_F)_+\calM(\beta)$ we consider successively the factors in (\ref{21}) from left to right.

Since $Q'_F$ is saturated in $\ZZ^d$, $\ZZ^d\cong Q'_F\times Q''_F$. 
Let $T'_F$ and $T''_F$ respectively be the tori corresponding to the free Abelian groups $Q'_F$ and $Q''_F$.
Then 
\[
\pi_F=\id_{T'_F}\times(T''_F\stackrel{\pi''_F}{\longrightarrow}{\rm point}).
\]
Let $t'=t'_1,\ldots,t'_{\dim F}$ and $t''=t''_{\dim F+1},\ldots,t''_{d}$ be coordinates on $T'_F$ and $T''_F$
respectively, where $t'$ are the coordinates corresponding to the
basis in $Q'_F$ for which $\gamma_F$ is induced by $K$.
Let $\beta'$ and $\beta''$ be the projections of $\beta$ onto $\CC Q'_F$ and $\CC Q''_F$ respectively and put $\calM'_F(\beta)=\calD_{T'_F}/\calD_{T'_F}\ideal{\del_{t'}t'+\beta'}$, $\calM''_F(\beta)=\calD_{T''_F}/\calD_{T''_F}\ideal{\del_{t''}t''+\beta''}$.
Then $\calM(\beta)=\calM'_F(\beta)\otimes_\CC\calM''_F(\beta)$ and hence
\begin{equation}\label{28}
(\pi_F)_+\calM(\beta)=\calM'_F(\beta)\otimes_\CC(\pi''_F)_+\calM''_F(\beta).
\end{equation}
As $t''$ is component-wise nonzero and since the fibers of $\pi''_F$ are affine,
$(\pi''_F)_+\calM''_F(\beta)$ can be computed by the Koszul
complex of left-multiplication by $\del_{t''}t''$, followed by
$(\pi''_F)_*$ \cite[IV.5.3.3]{Bor87}.
While $(\pi''_F)_*$ simply computes global sections, an elementary calculation shows that left-multiplication by $\del_{t_0}t_0$ on $D_{\CC^*}/D_{\CC^*}\ideal{\del_{t_0}t_0+\beta_0}$ is an isomorphism for $\beta_0\notin\ZZ$ and quasi-isomorphic to the complex $\CC\to\CC$ with zero differential otherwise.
We conclude 

\begin{prp}\label{22}
$(\pi_F)_+\calM(\beta)$ is nonzero only if $\beta''\in Q''_F$, and in the nonzero case represented by the complex $\calM'_F(\beta)\otimes_\ZZ\bigwedge^\bullet\ZZ^{d-\dim F}$ with zero differential.\qed
\end{prp}

We proceed to study the direct image of $\calM'_F(\beta)$ under
$\gamma_F$. 
In our chosen coordinate systems on $T'_F$ and $T_F$, it corresponds to the map of
semigroup rings $\CC[\ZZ F]\to\CC[Q'_F]$ given by
\[
\del_j\mapsto (\del'_j)^{k_j},\qquad 1\le i\le \dim F.
\]

In a single variable $\del_0$, the endomorphism $\ZZ\to\ZZ$ given by
$1\mapsto k$ induces 
$\del_0\mapsto (\del'_0)^k$ and so $k\del_0 t_0=\del_0't'_0$. 
The transfer module for the map $\zeta\colon T'_0\to T_0$ between the corresponding tori is $D_{T'_0\to T_0}=\zeta^* D_{T_0}=D_{T'_0}$ and thus $D_{T_0\from T'_0}=D_{T'_0}$ is free of rank one as right $D_{T'_0}$-module. 
As $T'_0$ and $T_0$ are affine, the direct image of $\calM'_0=D_{T'_0}/\ideal{\del'_0t'_0+\beta_0}$ is $\calM'_0$ itself, considered as $D_{T_0}$-module \cite[VI.5.4]{Bor87}. 
As $t_0$ is invertible and $\del'_0t'_0+\beta_0=k\del_0t_0+\beta_0$, 
\begin{eqnarray*}
\zeta_+\calM'_0&=&D_{T'_0}/\ideal{\del'_0t'_0+\beta_0}\\
&=&\bigoplus_{i=0}^{k-1}D_{T_0}(\del'_0)^i/\ideal{\del'_0t'_0+\beta_0}\cap D_{T_0}(\del'_0)^i\\
&\cong&\bigoplus_{i=0}^{k-1}D_{T_0}/\ideal{k\del_0t_0+\beta_0-i}
\end{eqnarray*}
since
${\del'_0}^i(\del'_0t'_0+\beta_0)=(\del'_0t'_0+\beta_0-i){\del'_0}^i$. 

It follows that in the coordinate system where $\ZZ F$ is generated by the image of $K$ we have 
\[
\gamma_+\calM'_F(\beta)=
\bigotimes_{i=1}^{\dim F}\left(\bigoplus_{r_i=0}^{k_i-1} \calM_{F_i}((\beta_i-r_i)/k_i)\right),
\]
$F_i$ being the $i$-th row of $F$ in the chosen basis.
On the level of locally constant solution sheaves, this corresponds to taking $K$-th roots.

The next two maps, $\iota_F$ and $\varpi_F$, present no difficulties
since all necessary work has already been done in
Section~\ref{17}. It remains to note that the embedding $X^*_F\into
X^*_A$ sends any $\calD_{X^*_F}$-module $\calM$ to
$\calM\otimes_\CC\CC[\bar F]$ in order to arrive at

\begin{thm}\label{32}
Let $F$ be a face of $A$, let $Q'_F=\QQ F\cap \ZZ^d\supseteq \ZZ
F$ be the saturation of its lattice in the ambient lattice, and let
$\ZZ^d=Q'_F\times Q''_F$ be a splitting.  
In the induced splitting $\CC^d=(\CC\otimes_\ZZ
Q'_F)\times (\CC\otimes_\ZZ Q''_F)$, write
$(\beta',\beta'')=\beta$ and $(E'_F,E''_F)=E$.

Choose coordinates in $Q'_F$ such that $\ZZ F\subseteq Q'_F$ is
generated by the image of the diagonal matrix
$K=\diag(k_1,\ldots,k_{\dim F})$.

The direct image $(\phi_F)_+\calM(\beta)$ is nonzero precisely when $\beta''$ is in $Q''_F$ (i.e., if $\beta''$ is the derivative of a torus character $T''_F\to\CC^*$). In the nonzero case, it is represented by 
\[
\CC[x_{\Bar F}]\otimes_\CC\bigoplus_\alpha 
K_\bullet(S_F[\del_F^{-1}],E'_F-\alpha)
\otimes_\ZZ\bigwedge^\bullet\ZZ^{d-\dim F}.
\]
Here, $\alpha$ runs through the $[Q'_F:\ZZ F]$ vectors for which
$\alpha_i=(\beta_i-r_i)/k_i$ and $r_i=0,\ldots,k_i-1$.\qed
\end{thm}

\section{Euler--Koszul homology and direct limits}\label{39}

Euler--Koszul homology has mostly been studied on the category of \emph{toric modules}, a class of finite $\ZZ^d$-graded modules determined by $S_A$. 
In this section we generalize ideas expanded in \cite{Oku06a}.

\begin{dfn}\label{31}
A \emph{weakly toric filtration} on a (possibly infinite) $\ZZ^d$-graded $R_A$-module $M$ is an exhaustive increasing filtration $\{M_\sss\}_{\sss\in\NN}$ by $\ZZ^d$-graded modules such that 
\begin{enumerate}
\item\label{31a} $M_{\sss+1}/M_\sss\cong S_{F_\sss}(\boldb_\sss)$ for some face $F_\sss$ of $A$ and some $\boldb_\sss\in\ZZ^d$ and
\item\label{31b} for all $\beta\in\ZZ^d$ the set $\{\sss\in\NN\mid-\beta\in\qdeg(M_{\sss+1}/M_\sss)\}$ is finite.
\end{enumerate}
A weakly toric filtration is \emph{toric} if it is finite; in that case the second condition is redundant.

If $M$ permits (weakly) toric filtrations then we call it \emph{(weakly) toric}.
Note that the first condition for a weakly toric filtration can be replaced by $M_{\sss+1}/M_\sss$ being toric.
The categories of toric and weakly toric $R_A$-modules are full subcategories of the category of $\ZZ^d$-graded $R_A$-modules with degree-preserving morphisms.
\end{dfn}

We remark that in condition (\ref{31b}) of Definition \ref{31} the union of $\{\sss\in\NN\mid-\beta\in\qdeg(M_{\sss+1}/M_\sss)\}$ over all $\beta\in\ZZ^d$ can be infinite.

\begin{prp}\label{34}
Any finitely generated $\ZZ^d$-graded $S_A$-module is toric.
A general $\ZZ^d$-graded $S_A$-module is weakly toric if the function $\boldb\mapsto\dim_\CC(M_\boldb)$ is bounded.
\end{prp}

\begin{proof}
The finite case follows from \cite[Ex.~4.7]{MMW05}. 

In the general case, we claim that the submodules $M_\sss$ of $M$ generated by all $\ZZ^d$-homogeneous elements of degree $\boldb$ with $\sum_{j=1}^d|b_j|\le \sss$ form a weakly toric filtration on $M$. 
First note that the filtration is exhaustive with toric quotients by Definition~\ref{31}.\eqref{31b} and the first sentence of the proof.
Next refine the filtration so that all filtration quotients are shifted face rings $M_{\sss+1}/M_\sss=S_{F_\sss}(\boldb_\sss)$ for suitable faces $F_\sss$ of $A$ and $\boldb_\sss\in\ZZ^d$.

Now suppose that
$\{\sss\in\NN\mid-\beta\in\qdeg(M_{\sss+1}/M_{\sss})\}$ is infinite
for some $\beta\in\ZZ^d$.  As the set of faces of $A$ is finite,
there exists a face $F$ with $-\beta\in\qdeg(S_F(\boldb_\sss))=\CC
F-\boldb_\sss$ for infinitely many $\sss$.  The difference of two such
$\boldb_s$ is then in $\CC F$ and hence in $\QQ F\cap\ZZ^d$.  The index
$[\QQ F\cap\ZZ^d:\ZZ F]$ is finite, hence there is an infinite
subsequence of the $\boldb_\sss$ with differences in the same coset of
$(\QQ F\cap\ZZ^d)/\ZZ F$.

For any choice of such $\boldb_{\sss_1},\ldots,\boldb_{\sss_k}$ the intersection $\bigcap_{j=1}^k(\NN F-\boldb_{\sss_j})$ is nonempty and for any element $\boldc$ of this intersection, $\boldc\in\NN F-\boldb_{\sss_j}=\deg(M_{\sss_j+1}/M_{\sss_j})$ for $j=1,\dots,k$.
As the function $N\mapsto\dim_\CC(N_\boldc)$ is additive in $N$,
$\dim_\CC(M_\boldc)>k$ and so $\boldb\mapsto
\dim_\CC(M_\boldb)$ cannot be bounded, in contradiction to the
hypothesis. 
\end{proof}

Part of our study of the Euler--Koszul functor applies in a quite general context: 
Let $(\frakS,\le)$ be a partially ordered set and pick a direct system $\frakM$ over $\frakS$ in the category of $\ZZ^d$-graded $R_A$-modules with degree-preserving morphisms,
\[
\frakM=\left(\{M_\sss\mid \sss\in\frakS\},\{\phi_{\sss,\sss'}\colon M_\sss\to M_{\sss'}\mid \sss\le\sss'\}\right).
\]
Then there is a $\ZZ^d$-graded direct limit
\begin{align}\label{43}
\phi_\sss\colon M_\sss\to\dlim_{\sss\in\frakS}M_\sss=:M.
\end{align}
We wish to discuss the Euler--Koszul complex $K_\bullet(M,\beta)$ on the potentially infinitely generated $R_A$-module $M$.
To begin with, note that the endomorphisms $E_i-\beta_i$ defined in (\ref{9}) induce endomorphisms of the direct system $D_A\otimes_{R_A}\frakM$, which allows to define $\dlim_{\sss\in\frakS}K_\bullet(M_\sss,\beta)$ and $\dlim_{\sss\in\frakS}H_\bullet(M_\sss,\beta)$. 
The natural maps $M_\sss\to M$ give rise to maps
\begin{align}
\dlim_{\sss\in\frakS}K_\bullet(M_\sss,\beta)&\to K_\bullet(M,\beta)\label{15}.
\end{align}
These induce
\begin{align}
\dlim_{\sss\in\frakS}H_\bullet(M_\sss,\beta)&\to H_\bullet(M,\beta)\label{16}
\end{align}
which are in general neither injective not surjective, cf.~Theorem~\ref{35}.\eqref{35a}.

Recall that $\frakS$ is called \emph{filtered} if for each $\sss',\sss''\in\frakS$ there exists $\sss\in\frakS$ with $\sss'\le\sss$ and $\sss''\le\sss$.
The following generalizes Definition \ref{41}.

\begin{dfn}\label{40}
The \emph{quasi-degrees} of a weakly toric $M$ are defined as
\[
\qdeg(M)=\bigcup_{\sss\in\NN}\qdeg(M_{\sss+1}/M_\sss).
\]
for any weakly toric filtration $\frakM=\{M_\sss\}_{\sss\in\NN}$. 

More generally, let $\frakS$ be filtered and pick a direct system $\frakM=\{M_\sss\}_{\sss\in\frakS}$ in the category of toric $R_A$-modules. 
With notation as in \eqref{43}, we define the \emph{quasi-degrees} of $M=\dlim\frakM$ as
\[
\qdeg(M)=\bigcup_{\sss\in\frakS}\qdeg(\phi_\sss(M_s)).
\]
Quasi-degrees are well-defined since toric modules are finitely generated. 
\end{dfn}
The following is the weakly toric version of \cite[Prop.~5.3]{MMW05}. 
\begin{thm}\label{35}
Let $\frakM$ and $M$ be as in Definition~\ref{40}.
\begin{asparaenum}
\item\label{35a} The map \eqref{15} is an isomorphism. If $\frakS$ is filtered then \eqref{16} is an isomorphism as well.
\item\label{35b} If $\frakS$ is filtered and if $\frakM$ is a direct system in the category of toric $R_A$-modules then $H_\bullet(M,\beta)=0$ if $-\beta\not\in\qdeg(M)$.
\item\label{35c} 
For any weakly toric module $M$, $H_\bullet(M,\beta)=0$ if and only if
$-\beta\not\in\qdeg(M)$.
\end{asparaenum}
\end{thm}

\begin{proof}\
\begin{asparaenum}

\item Direct limits commute with left-adjoint functors,
and hence with tensor products, in the sense that there is a natural
isomorphism
\[
\dlim_{\sss\in\frakS}K_i(M_\sss,\beta)=
\dlim_{\sss\in\frakS}\bigwedge^iD_A^d\otimes_{R_A}M_\sss\stackrel{\cong}{\longrightarrow}\bigwedge^iD_A^d\otimes_{R_A}M=
K_i(M,\beta)
\]
composing to the morphism in (\ref{15}).
If $\frakS$ is filtered then $\dlim_{\sss\in\frakS}$ is an exact functor and so (\ref{16}) is an isomorphism.

\item We may replace $M_\sss$ by its image $\phi_\sss(M_\sss)\subseteq M$ and assume that all $\phi_\sss$ are inclusions.
Then $-\beta\not\in\qdeg(M)\supseteq\qdeg(M_\sss)$ implies exactness of $K_\bullet(M_\sss,\beta)$ for all $\sss\in\frakS$ by \cite[Prop.~5.3]{MMW05} and hence exactness of $K_\bullet(M,\beta)$ by the filter condition.

\item Let $M=\bigcup_{\sss\in\NN}M_\sss$ be weakly toric and fix $\beta\in\CC^d$.
We continue to assume that $M_\sss=\phi_\sss(M_\sss)\subseteq M$ as in the previous part. 
By \cite[Prop.~5.3]{MMW05}, $-\beta\notin\qdeg(M_{\sss+1}/M_\sss)$ means that $0\to M_\sss\to M_{\sss+1}\to M_{\sss+1}/M_\sss\to 0$ induces a quasi-isomorphism $K_\bullet(M_\sss,\beta)\to K_\bullet(M_{\sss+1},\beta)$.
As $M_\sss$ is weakly toric, there is $\sss_\beta$ such that $-\beta\notin\qdeg(M_{\sss+1}/M_{\sss})$ for any $\sss\geq\sss_\beta$. 
Then $K_\bullet(M_{\sss_\beta},\beta)\cong K_\bullet(M,\beta)$ while $-\beta\notin\qdeg(M_{\sss_\beta})$ if and only if $-\beta\notin\qdeg(M)$. 
It thus suffices to show that $K_\bullet(M_{\sss_\beta},\beta)$ is exact if and only if $-\beta\notin\qdeg(M_{\sss_\beta})$, but that is \cite[Prop.~5.3]{MMW05}.

\end{asparaenum}
\end{proof}

\begin{rmk}\label{37}\
\begin{asparaenum}

\item\label{37a} The acyclicity statement in Proposition~\ref{1} can also be derived using Theorem~\ref{35}.\eqref{35c}.
The latter shows that (\ref{16}) is an isomorphism for
\[
\CC[\ZZ A]=S_A[\del_A^{-1}]=\dlim_{s\in\NN}\CC[\QQ_+A\cap\ZZ A]\cdot\del^{-s A}.
\]
As $\CC[\QQ_+A\cap\ZZ A]$ is normal and hence Cohen--Macaulay \cite{Hoc72}, $K_\bullet(\CC[\QQ_+A\cap\ZZ A],\beta)$ and hence $K_\bullet(\CC[\ZZ A],\beta)$ is a resolution for all $\beta$ \cite[Thm.~6.6]{MMW05}.

\item\label{37b} Localizations of weakly toric modules at subsets $\del_\tau$ of $\del_A$ are weakly toric since they arise as the union $M[\del_\tau^{-1}]=\dlim_{i\in\NN}\del_\tau^{-i}(M/\Gamma_{\del_\tau}(M))$.
It follows that the same holds for local cohomology modules $H^i_I(M)$ where $M$ is weakly toric and where $I$ is a monomial ideal of $R_A$. 
As a special case, $\qdeg(H^i_\frakm(M))=\deg(H^i_\frakm(M))$ as every finitely generated submodule of $ H^i_\frakm(M)$ is of finite length.

\item\label{37c} Left-multiplication by $\del_j$ gives a quasi-isomorphism on $K_\bullet(M,\beta)$ if and only if both $H_\bullet(\Gamma_{\del_j}(M),\beta)$ and $H_\bullet(H^1_{\del_j}(M),\beta)$ are zero. 
With $M=S_A$, the former is trivially zero, and
\[
\qdeg(H^1_{\del_j}(S_A))=\bigcup_{i=1}^\infty\left(\qdeg(S_A/\ideal{\del_j})-i\cdot\deg(\del_j)\right)
\]
by the Koszul interpretation of local cohomology. 
This yields an alternative proof for Theorem~\ref{6}.

\item\label{37d} For any weakly toric module $M$ and each $\beta\in\CC^d$, $H_\bullet(M,\beta)$ is a holonomic $D_A$-module. 
To see this, pick a weakly toric filtration $\{M_\sss\}_{\sss\in\NN}$ for $M$ and let $\sss_\beta$ such that $-\beta\in\qdeg(M_\sss)$ happens only for $\sss\le \sss_\beta$. 
Then, as in the proof of Theorem~\ref{35}.\eqref{35c}, $K_\bullet(M_{\sss_\beta},\beta)\to K_\bullet(M,\beta)$ is a quasi-isomorphism. 
Thus the claim follows from \cite[Prop.~5.1]{MMW05}.

\item\label{37e} The following refers to the notion of holonomic families, \cite[\S2]{MMW05}.

For weakly toric $M=\bigcup_{\sss\in\NN}M_\sss$, if $\bsb=b_1,\ldots,b_d$ are indeterminates, then the $D_A[\bsb]$-module
\[
\calM=(D_A[\bsb]\otimes_{R_A}M)/\ideal{E-\bsb}
\]
restricts to a holonomic $D_A$-module on each fiber of $X\times B\to B=\Spec(\CC[\bsb])$ while its rank module $\calM\otimes_{\CC[x_A]}\CC(x_A)$ is $\CC[\bsb](x_A)$-coherent if and only if $M$ is finitely generated over $R_A$.

Observe that $\bigcup_{\sss\in\NN}\qdeg(M_{\sss+1}/M_{\sss})$ is
locally in the analytic topology a finite union of subspaces of
$\CC^d$, namely a union of $\CC F+\ZZ^d$ over faces $F$ of $A$.
Thus, if $\sss_\beta$ is such that
$-\beta\not\in\qdeg(M_{\sss+1}/M_\sss)$ for all $\sss>\sss_\beta$ then
$K_\bullet(M,\gamma)\simeq K_\bullet(M_{\sss_\beta},\gamma)$ is induced
by the toric module $M_{s_\beta}$ for all $\gamma$ analytically near
$\beta$.

In particular, in the analytic topology, $\calM$ is locally on $B$ a holonomic family. 
One might hence call $\calM$ a \emph{weakly holonomic family}.
\end{asparaenum}
\end{rmk}

After these preparations we are ready to generalize Corollary \ref{36}.

\begin{cor}\label{12}
Let $\tau\subseteq A$ and fix $\beta\in\CC^d$. 
For $k\gg 0$, 
\[
R_A[\del_\tau^{-1}]\otimes_{R_A}H_\bullet(S_A,\beta)=H_\bullet(S_A,\beta+k\eps_\tau)(k\eps_\tau).
\]
More precisely, the above equality holds if $-(\beta+k\eps_\tau)\notin\qdeg(S_A/\ideal{\del^\tau})+\NN\eps_\tau$.
\end{cor}

\begin{proof}
We consider $S_A[\del_\tau^{-1}]=\dlim_k(S_A\cdot\del^{-k\tau})$ as direct limit of the $\ZZ^d$-graded modules $S_A\cdot\del^{-k\tau}$.
By (\ref{11}) and Theorem~\ref{35}.\eqref{35a}, 
\[
R_A[\del_\tau^{-1}]\otimes_{R_A}H_\bullet(S_A,\beta)
\cong H_\bullet(S_A[\del_\tau^{-1}],\beta)
=\dlim_kH_\bullet(S_A\cdot\del^{-k\tau },\beta)
\]
The natural maps
\begin{equation}\label{13}
H_\bullet(S_A,\beta+(k-1)\eps_\tau)\cong H_\bullet(S_A\cdot\del^{(-k+1)\tau},\beta)\to H_\bullet(S_A\cdot\del^{-k\tau },\beta)\cong H_\bullet(S_A,\beta+k\eps_\tau)
\end{equation}
are induced by right multiplication by $\del^\tau$ (or, alternatively, by $S_A(-(k-1)\eps_\tau)\to S_A(-k\eps_\tau)$). 
The maps \eqref{13} are enclosed in the long exact Euler--Koszul homology sequence by homology of $K_\bullet((S_A/\ideal{\del_\tau})\cdot\del^{-k\tau },\beta)\cong K_\bullet(S_A/\ideal{\del_\tau},\beta+k\eps_\tau)$.
By \cite[Prop.~5.3]{MMW05}, the latter is zero for $k\gg0$ since $-\qdeg(S_A/\ideal{\del_\tau})$ is bounded in $\eps_\tau$-direction.
Thus, 
\[
\dlim_kH_\bullet(S_A\cdot\del^{-k\tau },\beta)=H_\bullet(S_A\cdot\del^{-k\tau },\beta)=H_\bullet(S_A,\beta+k\eps_\tau)(k\eps_\tau)
\]
for $k\gg0$ as claimed.
\end{proof}

\section*{Acknowledgements}

We would like to thank Christine Berkesch for correcting an error in Remark~\ref{24}.


\bibliographystyle{amsalpha}
\bibliography{ekdi}

\end{document}